\documentclass[fleqn,reqno,11pt,a4paper,final]{amsart}

\usepackage[a4paper,left=30mm,right=30mm,top=30mm,bottom=30mm,marginpar=20mm]{geometry}
\usepackage{amsmath}
\usepackage{amssymb}
\usepackage{amsthm}
\usepackage{amscd}
\usepackage[ansinew]{inputenc}
\usepackage{cite}
\usepackage{bbm}
\usepackage{color}
\usepackage[english=american]{csquotes}
\usepackage[final]{graphicx}
\usepackage{hyperref}
\usepackage{calc}
\usepackage{mathptmx}
\usepackage{mathtools}
\usepackage{perpage}
\MakePerPage{footnote}

\usepackage{t1enc}
\numberwithin{equation}{section}

\newtheoremstyle{thmlemcorr}{10pt}{10pt}{\itshape}{}{\bfseries}{.}{10pt}{{\thmname{#1}\thmnumber{ #2}\thmnote{ (#3)}}}
\newtheoremstyle{thmlemcorr*}{10pt}{10pt}{\itshape}{}{\bfseries}{.}\newline{{\thmname{#1}\thmnumber{ #2}\thmnote{ (#3)}}}
\newtheoremstyle{remexample}{10pt}{10pt}{}{}{\bfseries}{.}{10pt}{{\thmname{#1}\thmnumber{ #2}\thmnote{ (#3)}}}
\newtheoremstyle{ass}{10pt}{10pt}{}{}{\bfseries}{.}{10pt}{{\thmname{#1}\thmnumber{ A#2}\thmnote{ (#3)}}}

\theoremstyle{thmlemcorr}
\newtheorem{theorem}{Theorem}
\numberwithin{theorem}{section}
\newtheorem{corollary}[theorem]{Corollary}
\newtheorem{conjecture}[theorem]{Conjecture}

\theoremstyle{remexample}

\newcommand{\Crm}{\mathrm{C}}
\newcommand{\Fcal}{\mathcal{F}}
\newcommand{\Ocal}{\mathcal{O}}
\newcommand{\Xcal}{\mathcal{X}}
\newcommand{\Bfrak}{\mathfrak{B}}
\newcommand{\Mbb}{\mathbb{M}}
\newcommand{\Nbb}{\mathbb{N}}
\newcommand{\Tbb}{\mathbb{T}}
\newcommand{\T}{\mathbb{T}^d}
\newcommand{\Zbb}{\mathbb{Z}}
\newcommand{\vr}{\varrho}
\newcommand{\dx}{{\rm d} {x}}
\newcommand{\dt}{{\rm d} t }
\newcommand{\dxdt}{\dx \ \dt}
\newcommand{\norm}[1]{\|#1\|}
\newcommand{\R}{\mathbb{R}}
\newcommand{\eps}{\varepsilon}
\newcommand{\inttT}{\int_0^T\int_\Tbb}
\newcommand{\tplus}{\Tbb\times\R_+}
\newcommand{\timT}{\Tbb\times (0,T)}
\newcommand{\partt}{\partial_t}
\newcommand{\parx}{\partial_x}
\renewcommand{\phi}{\varphi}
\DeclareMathOperator{\diverg}{div}
\DeclareMathOperator{\dist}{dist}
\DeclareMathOperator{\supp}{supp}

\begin{document}
\title{A tribute to conservation of energy for weak solutions}
\author{Tomasz D\k{e}biec \and Piotr Gwiazda \and Agnieszka \'{S}wierczewska-Gwiazda.
}
\address{\textit{Tomasz D\k{e}biec:} Institute of Applied Mathematics and Mechanics, University of Warsaw, Banacha 2, 02-097 Warszawa, Poland}
\email{t.debiec@mimuw.edu.pl}
\address{\textit{Piotr Gwiazda:} Institute of Mathematics, Polish Academy of Sciences, \'Sniadeckich 8, 00-656 Warszawa, Poland, and Institute of Applied Mathematics and Mechanics, University of Warsaw, Banacha 2, 02-097 Warszawa, Poland}
\email{pgwiazda@mimuw.edu.pl}
\address{\textit{Agnieszka \'{S}wierczewska-Gwiazda:} Institute of Applied Mathematics and Mechanics, University of Warsaw, Banacha 2, 02-097 Warszawa, Poland}
\email{aswiercz@mimuw.edu.pl}

\maketitle

\begin{abstract}
In this article we focus our attention on the principle of energy conservation within the context of systems of fluid dynamics. We give an overview of results concerning the resolution of the famous Onsager conjecture - which states regularity requirements for weak solutions of the incompressible Euler system to conserve energy. Further we survey results providing optimal sufficient regularity conditions for energy conservation for other balance laws: compressible Euler, Navier-Stokes, magnetohydrodynamics and general conservation laws.
\end{abstract}

\section{Introduction}
The general global existence theory for systems of equations governing the motion of fluids is still an open problem, considered by many to be one of the most difficult problems in the field. It is however usually quite straightforward to assert that a classical solution, should it exist, will conserve the physical energy. However, basing on empirical observations, it was long expected that in a fully turbulent flow one might observe dissipation of energy. This was shown to be true by Scheffer \cite{scheffer} and Shnirelmann \cite{shnirel} - they constructed weak solutions for the incompressible Euler system which do not conserve kinetic energy. It is however a crucial feature of such solutions that they are irregular. Within the context of incompressible flows it was suggested by Lars Onsager \cite{On1949} in late forties that there exists a threshold regularity beyond which energy is conserved. The positive direction of his statement was resolved around the time that the examples of Scheffer and Shnirelmann were published. However the construction of highly regular dissipative solutions remained open until very recently. Furthermore, Onsager-type statements focused much interest in the context of other systems, in particular for compressible fluid dynamics.\\
In this paper we give an overview of results regarding energy conservation for a number of systems of fluid dynamics, including the classical Onsager statement and its refinements as well as the more recent developments.  
\subsection{Local vs. global conservation}
%Let us observe the local nature of the result obtained in the previous section.
When considering energy conservation for solutions of physical systems, one can mean one of the two things: either some local energy conservation principle or an energy identity of global nature. Let us consider the incompressible Euler system:
\begin{equation}\label{eulerintro}
\partt u + \diverg(u\otimes u) + \nabla p = 0,\;\;\; \diverg{u}=0.
\end{equation}
It is known, see discussion is Section \ref{secEuler}, that classical solutions to ~\eqref{eulerintro} must satisfy the following \emph{local} law of energy conservation
\begin{equation}\label{localenergyintro}
    \partt \frac{1}{2}|u|^2 + \diverg\left(u(\frac{1}{2}|u|^2+p)\right) = 0.
\end{equation}
From this identity one can deduce the \emph{global} identity
\begin{equation}\label{globalenergyintro}
    \int_{\R^d}\frac{1}{2}|u(x,t)|^2\ \dx = \int_{\R^d}\frac{1}{2}|u(x,s)|^2\ \dx
\end{equation}
for any $t,s\geq 0$. To see this, we integrate the local identity in time and space. However, sufficient decay of $u$ and $p$ has to be assumed at infinity. The last equality can be written as
\[\frac{\mathrm{d}}{\mathrm{d}t}\left(\int_{\R^d}\frac{1}{2}|u(x,t)|^2\ \dx\right) = 0,
\]
which will be referred to as the total energy balance.
The main issue of this article is to investigate whether \emph{weak} solutions of ~\eqref{eulerintro}, and other systems, satisfy any of the above energy conservation laws. Clearly one can not hope for the global energy balance to imply the local one. In this sense the local form of energy conservation is a stronger result. However, it is not clear that in general ~\eqref{localenergyintro} shall imply ~\eqref{globalenergyintro}, especially when a closed temporal domain is considered. In particular, some weak continuity in time might have to be assumed.
\subsection{Organization of the paper}
We begin the discussion with a simple case of a scalar conservation law, which serves as a toy model to display the useful technical machinery. Then in Sections \ref{secEuler} and \ref{secdissipative} we discuss the full resolution of Onsager's original conjecture for the incompressible Euler equations. In section \ref{secOnsagertype} we survey the available results on sufficient regularity of conservative weak solutions to other systems: compressible Euler, incompressible and compressible Navier-Stokes equations and ideal magnetohydrodynamics. Finally in Section \ref{secgeneral} we present a recent result of Gwiazda et al. \cite{GMSG} concerning general conservation laws.

\section{Basic commutator estimate}\label{secSimple}
In this section we wish to display the working of the main analytical tools that are usually employed in the context of showing conservation of energy for systems of fluid dynamics - namely smoothing of the balance equations and estimating commutator errors due to nonlinearities. This method was applied by Constantin et al. in \cite{ConstETiti} for the incompressible Euler, where the nonlinear term has a bilinear structure. Later an extension to the case of more general nonlinearities was presented by Feireisl et al. in \cite{FeGSGW}. This allowed treating the pressure term in the compressible Euler system. We will demonstrate here the estimates of the latter.
To this end we consider the very simple case of a single conservation law in one space dimension.
%We begin by considering the scalar conservation law
\begin{equation}\label{simpleConservlaw}
%\begin{aligned}
    \partial_t u + \parx f(u) = 0,
%\end{aligned}
\end{equation}
where $u:\timT\to\R$, $T>0$, is a conserved physical quantity and $f:\R\to\R$ is the flux of this quantity. We will assume that the flux function is twice continuously differentiable.\\
To this equation one can associate a pair $(\eta, q)$ of continously differentiable real-valued functions on $\R$ such that for each $u\in\R$
\begin{equation}\label{entropyfluxrelation}
\eta'(u) = f'(u)q'(u).
\end{equation}
If $u$ is a $\Crm^1(\timT)$ solution of ~\eqref{simpleConservlaw}, then multiplying ~\eqref{simpleConservlaw} by $\eta'(u)$ and invoking the chain rule yields
\begin{equation}\label{simpleConserventropy}
\partt\eta(u) + \parx q(u) = 0.
\end{equation}
Thus a classical solution of the initial conservation law satisfies also the additional conservation law ~\eqref{simpleConserventropy}, where the conserved quantity $\eta$ is called an \emph{entropy} for ~\eqref{simpleConservlaw} with \emph{entropy flux} $q$. 
However weak solutions of ~\eqref{simpleConservlaw} can be found, which do not satisfy ~\eqref{simpleConserventropy}. For example, if $f(u)=\frac{u^2}{2}$, then taking $\eta(u) = u^3$ and $q(u)=\frac{3}{4}u^4$, one can see a violation of ~\eqref{simpleConserventropy} by the steady weak solution 
 \[u(x,t) = 
                \begin{cases}
                1, &x<\frac{1}{2},    \\
                0, &x\geq\frac{1}{2}.
                \end{cases}
                \]
By a weak solution we mean a locally integrable function $u$, which satisfies ~\eqref{simpleConservlaw} in the sense of distributions on $\timT$.
In light of the non-uniqueness of such distributional solutions, an admissibility requirement can be posed for a solution to satisfy ~\eqref{simpleConserventropy} at least as an inequality for a convex entropy $\eta$. It is well-known that there exists a unique admissible weak solution to ~\eqref{simpleConservlaw}, cf. \cite{Kru}.\\
The goal of this section is to answer the following question: \emph{What is the minimal regularity requirement for a weak solution to ~\eqref{simpleConservlaw} so that it conserves entropy?} A similar question was asked by Lars Onsager for the incompressible Euler equations, which will be discussed in a later section. His conjecture was stated in terms of an appropriate H\"{o}lder continuity exponent. It was however observed by Constantin et al. in \cite{ConstETiti} that an appropriate setting in looking for sharp regularity conditions is provided by the Besov spaces. We define the space $B_p^{\alpha,\infty}(\T)$ to be
\begin{equation}
B_p^{\alpha,\infty}(\T)\coloneqq\left\{u\in L^p(\T)\;\Big |\;\norm{u}_{B_p^{\alpha,\infty}(\T)}<\infty \right\}
\end{equation}
where
\begin{equation}\label{besovnorm}
\norm{u}_{B_p^{\alpha,\infty}(\T)}\coloneqq\norm{u}_{L^p(\T)}+\sup\limits_{y\in\T}\frac{\norm{u(\cdot+y)-u(\cdot)}_{L^p(\T\cap(\T-y))}}{|y|^\alpha}.
\end{equation}
Notice that clearly $C^\alpha\subset B_p^{\alpha,\infty}$. Later we will describe a characterisation of Besov spaces using instruments of harmonic analysis.
We will now prove the following theorem
\begin{theorem}\label{thmsimple}
Let $u$ be a solution of ~\eqref{simpleConservlaw} in the sense of distributions. If $u$ belongs to $L^3(0,T;B_3^{\alpha,\infty}(\Tbb))$ for some $\alpha>\frac{1}{3}$, then it satisfies ~\eqref{simpleConserventropy} in the sense of distributions for any $\Crm^2$ entropy $\eta$ and corresponding entropy flux $q$.
\end{theorem}
%%%%%%%%%%%%
%%%%%%%%%%%%
\begin{proof}
Let $\theta\in\Crm_c^\infty(\Tbb)$ be a non-negative symmetric function with $\int_{\Tbb}\theta(x)\ \dx = 1$. We set $\theta^\eps = \frac{1}{\eps}\theta(\frac{x}{\eps})$. For a locally integrable function $u$ on $\timT$ we set 
\[u^\eps(x,t) = (u *_x \theta^\eps)(\cdot,t) = \int_{\Tbb}\theta^\eps(x')u(x-x',t)\ \mathrm{d}x'.
\]
This function is clearly well-defined on the set $\left\{(x,t)\in\timT\;|\;\dist\left((x,t),\partial(\timT)\right)>\eps\right\}$.\\
From the definition ~\eqref{besovnorm} one can readily deduce the following estimates
%\begin{aligned}
%\begin{equation}\label{besovshift}
%\norm{u(\cdot+y,t)-u(\cdot,t)}_{L^3(\Tbb)}&\leq C\;|y|^\alpha\norm{u(\cdot,t)}_{B_p^{\alpha,\infty}(\Tbb)}
%\end{equation}
%\begin{equation}\label{besovapprox}
%\norm{u^\eps(\cdot,t)-u(\cdot,t)}_{L^3(\Tbb)}&\leq C\eps^\alpha\norm{u(\cdot,t)}_{B_p^{\alpha,\infty}(\Tbb)}
%\end{equation}
%\begin{equation}\label{besovgrad}
%\norm{\nabla u^\eps(\cdot,t)}_{L^3(\Tbb)}&\leq %C\eps^{\alpha-1}\norm{u(\cdot,t)}_{B_p^{\alpha,\infty}(\Tbb)}
%\end{equation}
%.
%\begin{equation}
\begin{align}
&\norm{u(\cdot+y,t)-u(\cdot,t)}_{L^3(\Tbb)}\leq C\;|y|^\alpha\norm{u(\cdot,t)}_{B_p^{\alpha,\infty}(\Tbb)}\label{besovshift}\\ 
&\norm{u^\eps(\cdot,t)-u(\cdot,t)}_{L^3(\Tbb)}\leq C\eps^\alpha\norm{u(\cdot,t)}_{B_p^{\alpha,\infty}(\Tbb)} \label{besovapprox}\\
&\norm{\nabla u^\eps(\cdot,t)}_{L^3(\Tbb)}\leq C\eps^{\alpha-1}\norm{u(\cdot,t)}_{B_p^{\alpha,\infty}(\Tbb)} \label{besovgrad}
\end{align}
for any $t\in(0,T)$.\\
%\end{equation}
Mollifying equation ~\eqref{simpleConservlaw} in space yields
\begin{equation}\label{smoothsimpleconservlaw}
\partt u^\eps + \parx f^\eps(u) = 0,
\end{equation}
where we have used the fact that mollifiers commute with derivatives and we denote
\[f^\eps(u(x,t)) = \int_{\Tbb}f(u(x-x',t))\theta^\eps(x')\ \mathrm{d}x'.
\]
Equation ~\eqref{smoothsimpleconservlaw} can be written as 
\begin{equation}\label{simplesmoothcommutator}
\partt u^\eps + \parx f(u^\eps) = \parx\left(f(u^\eps) - f^\eps(u)\right).
\end{equation}
Now let $(\eta,q)$ be an entropy-entropy flux pair as described above with $\eta\in\Crm^2(\R)$. Further let $\phi\in\Crm_c^\infty(\timT)$ and choose $\eps>0$ small enough so that $\supp\phi\subset\Tbb\times(\eps,T-\eps)$. We multiply the last equation by $\eta'(u^\eps)\phi$ and integrate over time and space to get
\begin{equation}
\begin{aligned}
\inttT\partt u^\eps\eta'(u^\eps)\phi\ \dxdt + \inttT&\parx (f(u^\eps))\eta'(u^\eps)\phi\ \dxdt\\
&= \inttT\parx\left(f(u^\eps) - f^\eps(u)\right)\eta'(u^\eps)\phi\ \dxdt.
\end{aligned}
\end{equation}
Due to ~\eqref{entropyfluxrelation} the left-hand side of the last equation can be easily seen to be equal to
\[\inttT\left[\partt\eta(u^\eps)+\parx q(u^\eps)\right]\phi\ \dxdt.
\]
We will now show that, under our regularity conditions on $u$, the integral
\[R_\eps\coloneqq\inttT\parx\left(f(u^\eps) - f^\eps(u)\right)\eta'(u^\eps)\phi\ \dxdt
\]
converges to zero as $\eps\to 0^+$. To this end we observe that by Taylor's theorem we have
\begin{equation}\label{Taylor1}
\left|f(u^\eps(x,t))-f(u(x,t))-f'(u(x,t))(u^\eps(x,t)-u(x,t))\right|\leq C(u^\eps(x,t)-u(x,t))^2
\end{equation}
where, importantly, the constant $C$ does not depend on the choice of $x$ and $t$ (it depends only on the norm $\norm{f}_{\Crm^2(\R)}$). Similarly
\begin{equation}\label{Taylor2}
\left|f(u(y,t))-f(u(x,t))-f'(u(x,t))(u(y,t)-u(x,t))\right|\leq C(u(y,t)-u(x,t))^2.
\end{equation}
Mollification of the last inequality with respect to $y$ yields, by virtue of Jensen's inequality
\begin{equation}\label{Taylor3}
\left|f^\eps(u(x,t))-f(u(x,t))-f'(u(x,t))(u^\eps(x,t)-u(x,t))\right|\leq C(u(\cdot,t)-u(x,t))^2*_y\theta^\eps.
\end{equation}
Combining ~\eqref{Taylor1} and ~\eqref{Taylor3} and using the triangle inequality we deduce the estimate
\begin{equation}\label{nonlinearestimate}
\left|f(u^\eps(x,t))-f^\eps(u(x,t))\right|\leq C\left[(u^\eps(x,t)-u(x,t))^2 + (u(\cdot,t)-u(x,t))^2*_y\theta^\eps\right].
\end{equation}
Now we consider again the error term $R_\eps$. Integrating by parts we obtain
\begin{equation*}
    R_\eps = -\inttT\left(f(u^\eps)-f^\eps(u)\right)\left(\eta''(u^\eps)\parx u^\eps\phi+\eta'(u^\eps)\parx\phi\right)\ \dxdt.
\end{equation*}
The inner integral can be estimated by
\begin{equation*}
\begin{aligned}
\int_\Tbb&|f(u^\eps)-f^\eps(u)||\eta''(u^\eps)\parx u^\eps\phi|\ \dx + \int_\Tbb|f(u^\eps)-f^\eps(u)||\eta'(u^\eps)\parx\phi|\ \dx\\
&\leq \norm{\eta}_{\Crm^2(\R)}\norm{\phi}_{\Crm(\timT)}\underbrace{\int_\Tbb|f(u^\eps)-f^\eps(u)||\parx u^\eps|\ \dx}_{\text{$I_1$}}\\
&\hspace{0.5cm}+ \norm{\eta}_{\Crm^1(\R)}\norm{\phi}_{\Crm^1(\timT)}\underbrace{\int_\Tbb|f(u^\eps)-f^\eps(u)|\ \dx}_{\text{$I_2$}}.
\end{aligned}
\end{equation*}
Using the estimate ~\eqref{nonlinearestimate} we can calculate
\begin{equation*}
\begin{aligned}
I_2 &\leq C\left[\int_\Tbb|u^\eps(x,t)-u(x,t)|^2\ \dxdt + \int_\Tbb|(u(\cdot,t)-u(x,t))^2 *_y\theta^\eps|\ \dx\right] \\
&\leq C\left(\int_\Tbb|u^\eps(x,t)-u(x,t)|^3\ \dx\right)^{\frac{2}{3}} + C\int_\Tbb\;\sup\limits_{|y|<\eps}{|u(x-y,t)-u(x,t)|^2}\ \dx \\
&\leq C\left(\int_\Tbb|u^\eps(x,t)-u(x,t)|^3\ \dx\right)^{\frac{2}{3}} + C\sup\limits_{|y|<\eps}\left(\int_\Tbb|u(x-y,t)-u(x,t)|^3\ \dx\right)^{\frac{2}{3}} \\
&\leq C\left(\norm{u^\eps(\cdot,t)-u(\cdot,t)}^2_{L^3(\Tbb)} + \sup\limits_{|y|<\eps}\norm{u^\eps(\cdot-y,t)-u(\cdot,t)}^2_{L^3(\Tbb)}\right)\\
&\leq C\eps^{2\alpha}\norm{u(\cdot,t)}^2_{B_3^{\alpha,\infty}(\Tbb)}.
\end{aligned}
\end{equation*}
%Since no of the constants in the above string of inequalities depends on $\eps$, we see that $I_2\to 0$ as $\eps\to 0^+$.
Similar estimation can be carried out for $I_1$. For brevity we will ignore the term of ~\eqref{nonlinearestimate} involving the convolution. As seen above this term produces estimates of the same type as the other term. We therefore have
\begin{equation*}
\begin{aligned}
I_1 &\leq\int_\Tbb|u^\eps(x,t)-u(x,t)|^2|\parx u^\eps(x,t)|\ \dx \leq \norm{u^\eps(\cdot,t)-u(\cdot,t)}^2_{L^3(\Tbb)}\norm{\parx u^\eps(\cdot,t)}_{L^3(\Tbb)}\\
&\leq C\eps^{2\alpha}\eps^{\alpha-1}\norm{u(\cdot,t)}^3_{B_3^{\alpha,\infty}(\Tbb)}.
\end{aligned}
\end{equation*}
Both of the above estimates hold for any time $t\in(0,T)$ and the constants do not depend on $\eps$. Integrating w.r.t. time we therefore see that $R_\eps\to 0$ as $\eps\to 0^+$ provided $3\alpha-1>0$, i.e. $\alpha>\frac{1}{3}$.
%\begin{equation}\label{finalestimate}
%    |R_\eps|\leq C(\eps^{2\alpha}+\eps^{3\alpha+1})\norm{u(\cdot,t)}_{L^B_3^{\alpha,\infty}(\Tbb)}
%\end{equation}
%The right-hand side of the last inequality converges to zero as $\eps\to 0^+$ provided $3\alpha-1>0$, i.e. $\alpha>\frac{1}{3}$.\\
We have shown that the right-hand side of ~\eqref{simplesmoothcommutator} vanishes as $\eps\to 0^+$. It follows from standard approximating properties of mollifiers that passage to the limit in the left-hand side of the same equality is possible. We therefore obtain
\begin{equation}
\inttT\left[\partt\eta(u)+\parx q(u)\right]\phi\ \dxdt = 0,
\end{equation}
which holds for any $\phi\in\Crm_c^\infty(\timT)$. 
\end{proof}

\section{Incompressible Euler}\label{secEuler}
In this section we focus our attention on the incompressible Euler system
\begin{equation}\label{incompressEuler}
\begin{aligned}
\partt u + \diverg(u\otimes u) + \nabla p &= 0,\\
\diverg{u} &= 0,
\end{aligned}
\end{equation}
where the space domain is either the periodic domain $\T$ or the whole space $\R^d$ and the temporal domain is the interval $[0,T]$.
These equations were the subject of the original Onsager conjecture from 1949, cf. \cite{On1949}. If $u$ is a classical solution of ~\eqref{incompressEuler}, then multiplying the balance equation by $u$ we obtain
\[\frac{1}{2}\partt |u|^2 + \frac{1}{2}u\cdot\nabla|u|^2 + u\cdot\nabla p = 0.
\]
Integrating the last equality over the space domain $\Omega$, integrating the last two terms by parts and using the incompressibility condition yields
\[
\frac{\mathrm{d}}{\mathrm{d}t}\int_\Omega\frac{1}{2}|u(x,t)|^2\ \dx = 0.  
\]
Consequently, integrating over time in $(0,t)$, gives
\begin{equation}\label{energyconservincompressEuler}
\int_{\Omega}\frac{1}{2}|u(x,t)|^2\ \dx = \int_\Omega\frac{1}{2}|u(x,0)|\ \dx
\end{equation}
for any $t\in[0,T]$. Thus the principle of conservation of energy holds for classical solutions of ~\eqref{incompressEuler}. Similarly as in Section \ref{secSimple} however, if $u$ is a weak solution, then ~\eqref{energyconservincompressEuler} might not hold. Technically, the problem is that $u$ might not be regular enough to justify integration by parts in the above derivation.\\ 
Motivated by the laws of turbulence Onsager postulated that there is a critical regularity for a weak solution to be a conservative one. Rigorously he stated the following conjecture
\begin{conjecture}\label{onsagerconjecture}
Let $u$ be a weak solution of ~\eqref{incompressEuler}.
\begin{itemize}
    \item If $u\in\Crm^\alpha$ with $\alpha>\frac{1}{3}$, then the energy is conserved.
    \item For any $\alpha<\frac{1}{3}$ there exists a weak solution $u\in C^{\alpha}$ which does not conserve the energy.
\end{itemize} 
\end{conjecture}
The first result concerning the first part of the conjecture was due to Eyink in \cite{eyink}. This was however not a proof of Onsager's conjecture, as it required stronger regularity. Concretely Eyink required that for a.e. time the velocity field $u$ belongs to the space $\Crm_*^\alpha(\T)$, $\alpha>\frac{1}{3}$, of functions, whose Fourier coefficients satisfy the following summability condition
\[\sum\limits_{k\in\Zbb^d}|k|^\alpha|\hat{u}(k)| < \infty.
\]
This condition can be easily seen to imply H\"{o}lder continuity of order $\alpha$, i.e. $\Crm_*^\alpha\subset\Crm^\alpha$. However there is no equality. For example the function 
\[u(x) = \sum\limits_{n=1}^\infty\frac{\cos{(3^nx)}}{3^{\alpha n}}
\]
is of class $\Crm^\alpha(\Tbb)$ for any $0<\alpha\leq 1$, but not $\Crm_*^\alpha(\Tbb)$.\\
The first full proof of the sufficient condition predicted by Onsager was published in 1994 by Constantin, E and Titi, cf. \cite{ConstETiti}, who actually strengthened the original statement by assuming spatial regularity of class $B_3^{\alpha,\infty}(\T)$ with $\alpha>\frac{1}{3}$. As observed before the Besov space $B_3^{\alpha,\infty}(\T)$ contains $C^\alpha(\T)$ as a proper subspace. We recall here the full proof of the result in \cite{ConstETiti} for reader's convenience.
\begin{theorem}
Let $u\in L^3([0,T],B_3^{\alpha,\infty}(\T))\cap\Crm([0,T],L^2(\T))$ be a weak solution of the incompressible Euler system. If $\alpha>\frac{1}{3}$, then
\[\int_{\T}\frac{1}{2}|u(x,t)|^2\ \dx = \int_{\T}\frac{1}{2}|u(x,0)|^2\ \dx
\]
for each $t\in[0,T]$.
\end{theorem}
\begin{proof}
Following \cite{ConstETiti} we observe that mollification of the weak solution in time poses no technical difficulty, so for brevity of notation we ignore it in the proof. We use notation as in the proof of Theorem \ref{thmsimple}. After mollifying equation ~\eqref{incompressEuler} in space we obtain
\begin{equation}
\partt u^\eps + \diverg(u^\eps\otimes u^\eps) + \nabla p^\eps = \diverg\left(u^\eps\otimes u^\eps - (u\otimes u)^\eps\right).
\end{equation}
Multiplying this equation by $u^\eps$ and integrating over the space domain $\T$ we get for any time $t$
\[
\int_{\T}\partt\left(\frac12 |u^\eps(x,t)|^2\right)\ \dx = \int_{\T}\left[(u\otimes u)^\eps-u^\eps\otimes u^\eps\right]:\nabla u^\eps\ \dx.
\]
Integrating in turn with respect to time on $[0,\tau)$ we obtain
\begin{equation}\label{smoothedenergy}
\int_{\T}\frac12|u^\eps(x,\tau)|^2\ \dx - \int_{\T}\frac12|u^\eps(x,0)|^2\ \dx = \int_0^\tau\int_{\T}\left[(u\otimes u)^\eps-u^\eps\otimes u^\eps\right]:\nabla u^\eps\ \dxdt.
\end{equation}
Clearly  we can pass with $\eps$ to zero on the left-hand side of the last equality to obtain
\[\int_{\T}\frac12|u(x,\tau)|^2\ \dx - \int_{\T}\frac12|u(x,0)|^2\ \dx.
\]
Therefore to conclude the proof it is enough to show that the right-hand side of ~\eqref{smoothedenergy} converges to zero as $\eps\to 0^+$. To this end we observe the following identity
\begin{equation}\label{commutatoridentity}
(u\otimes u)^\eps(x,t)-(u^\eps\otimes u^\eps)(x,t) = r_\eps(x,t) - \left[(u-u^\eps)\otimes(u-u^\eps) \right](x,t),
\end{equation}
which holds for any $(x,t)\in\T\times[0,T]$, where
\[r_\eps(x,t)\coloneqq\int_{\T}\theta^\eps(x')\left[(u(x-x',t)-u(x,t))\otimes(u(x-x',t)-u(x,t))\right]\ \mathrm{d}x'.
\]
Substituting ~\eqref{commutatoridentity} into ~\eqref{smoothedenergy} we obtain two terms to be estimated
\begin{equation*}
    \begin{aligned}
  \Big|\int_{\T}&[(u\otimes u)^\eps-u^\eps\otimes u^\eps]:\nabla u^\eps\ \dx \Big| \leq \int_{\T}|u^\eps-u|^2|\nabla u^\eps|\ \dx + \int_{\T}|r_\eps||\nabla u^\eps|\ \dx\\
  &\leq\norm{u^\eps(\cdot,t)-u(\cdot,t)}^2_{L^3(\T)}\norm{\nabla u^\eps(\cdot,t)}_{L^3(\T)} + \norm{r_\eps(\cdot,t)}_{L^{\frac32}}\norm{\nabla u^\eps(\cdot,t)}_{L^3(\T)}
    \end{aligned}.
\end{equation*}
We make note of the following estimate
\begin{equation*}
    %\begin{aligned}
    |r_\eps(x,t)|^{\frac32}\leq |\theta^\eps(x')|^{\frac32}|u(x-x',t)-u(x,t)|^3\ \mathrm{d}x'\leq C\sup\limits_{|x'|\leq\eps}|u(x-x',t)-u(x,t)|^3.
    %\end{aligned}.
\end{equation*}
It follows, using ~\eqref{besovshift}, that
\begin{equation*}
    \norm{r_\eps(\cdot,t)}_{L^{\frac32}}\leq C\eps^{2\alpha}\norm{u}_{B_3^{\alpha,\infty}(\T)}^2.
\end{equation*}
Consequently, using ~\eqref{besovapprox} and ~\eqref{besovgrad}, we have the estimate
\begin{equation*}
\Big|\int_{\T}[(u\otimes u)^\eps-u^\eps\otimes u^\eps]:\nabla u^\eps\ \dx \Big| \leq C\eps^{2\alpha}\eps^{\alpha-1}\norm{u(\cdot,t)}_{B_3^{\alpha,\infty}}.
\end{equation*}
Integrating the last inequality in time we obtain
\begin{equation*}
\Big|\int_0^\tau\int_{\T}[(u\otimes u)^\eps-u^\eps\otimes u^\eps]:\nabla u^\eps\ \dxdt \Big| \leq C\eps^{2\alpha}\eps^{\alpha-1}\norm{u(\cdot,t)}_{L^3(0,T;B_3^{\alpha,\infty})}.
\end{equation*}
\end{proof}
There have since been several results which refine the above theorem. Duchon and Robert have shown in \cite{DuRo} that energy conservation ~\eqref{energyconservincompressEuler} holds true provided the velocity field satisfies the following integral condition
\begin{equation}
\int_{\T}|u(x+y,t)-u(x,t)|^3\ \dx \leq C(t)|y|\sigma(|y|),\;\forall y\in\T
\end{equation}
where $C$ is an integrable function on $[0,T]$ and $\sigma(a)\to 0$ as $a\to 0$.
In fact they have shown that under the above assumption a local energy equality
\begin{equation}
\partt\left(\frac{1}{2}|u|^2\right) + \diverg\left(u(\frac{1}{2}|u|^2+p)\right) = 0
\end{equation}
holds in the sense of distributions. A further refinement was provided by Cheskidov et al. in ~ \cite{cheskidov}, where the space $B_3^{\frac{1}{3},c(\Nbb)}(\R^d)$ is introduced via Littlewood-Paley decomposition. Since this framework proved to be very potent and is commonly used in this context, we introduce here the Littlewood-Paley decomposition and give a definition of a Besov norm equivalent to the one used in Section \ref{secSimple}.\\
Let $\psi\in\Crm_0^\infty(B(0,1))$ be a non-negative radial function, defined on the open unit ball in $\R^d$, such that $\psi(r) = 1$, for $r\leq\frac12$. Define $\zeta(\xi) = \psi(\frac{\xi}{2})-\psi(\xi)$. Further, denoting by $\mathcal{F}$ the Fourier transform on $\R^d$, we define the following operators
\begin{align*}
h &= {\Fcal}^{-1}\zeta\quad {\rm and}\quad \tilde{h} ={\Fcal}^{-1}\psi, \\
\Delta_q u & = \Fcal^{-1}(\zeta(2^{-q}\xi)\Fcal u) = 2^{qd}\int h(2^q y)u(x-y)\ \mathrm{d}y, \;\; q\ge 0,\\
\Delta_{-1}u & =\Fcal^{-1}(\psi(\xi)\Fcal u) = \int \tilde{h}(y)u(x-y)\ \mathrm{d}y.
\end{align*}
For $N\in\Nbb$ we define the operator $S_N$ by
\[S_N = \sum\limits_{q=-1}^{N}\Delta_q.
\]
Then the inhomogeneous Besov space $B_p^{s,r}(\R^d)$ is defined to be the space of all tempered distributions $u$ for which the following norm is finite
\begin{equation}\label{besovnorm2}
\norm{u}_{B_p^{s,r}}\coloneqq \norm{\Delta_{-1}u}_{L^p} + \norm{\{2^{qs}\norm{\Delta_q u}_{L^p}\}_{q\in\Nbb}}_{l^p(\Nbb)}.
\end{equation}
Furthermore one can define the space $B_3^{\frac{1}{3},c(\Nbb)}(\R^d)$ of tempered distributions for which
\[\lim\limits_{q\to\infty}\left(2^{\frac{q}{3}}\norm{\Delta_q u}_{L^3}\right) = 0
\]
together with the norm inherited from $B_3^{\frac{1}{3},\infty}$. A similar characterisation is given for $L^p$ spaces in \cite{Bask}, see also \cite{BahouriCheminDanchin} for a detailed discussion on Besov spaces.\\
This framework is used in \cite{cheskidov} to show global energy conservation for velocities in $L^3(0,T;B_3^{\frac{1}{3},c(\Nbb)}(\R^3))$, and in particular in $L^3(0,T;B_3^{\frac{1}{3},q}(\R^3)$ for any $q\in[1,\infty)$. Shortly after Shvydkoy proved a local energy conservation under the assumption 
\[\lim\limits_{|y|\to 0}\frac{1}{|y|}\int_{\R^d\times[0,T]}|u(x+y,t)-u(x,t)|^3\ \dxdt = 0,
\]
which is actually equivalent to the one proposed in \cite{cheskidov}, see \cite{shvydkoy} for details.\\
So far each result on sufficient regularity conditions for conservative solutions to ~\eqref{incompressEuler} assumed either periodic boundary conditions or dealt with the whole space domain $\R^d$. Recently however first results appeared treating the case of a bounded domain. Bardos and Titi have proved the following theorem, cf. \cite{BarTiti}
\begin{theorem}
Let $\Omega\subset\R^d$ be a bounded domain with  $C^2$  boundary, $\partial\Omega$; and let  $(u(x,t),p(x,t))$ be a weak solution of the incompressible Euler equations in $\Omega\times(0,T)$, i.e.,
%\begin{subequations}
\begin{equation}
 u\in L^\infty((0,T),L^2(\Omega))\,, \quad \nabla \cdot u=0\quad  \hbox{in} \quad \Omega\times(0,T)\,,\,  \quad u\cdot n =0 \quad \hbox{on} \quad \partial \Omega\times(0,T)\,,
 \end{equation}
%\begin{equation}
%\nabla \cdot u=0 \quad \hbox{ in} \quad  \Omega\quad
% \end{equation}
and for every test  vector field  $ \Psi (x,t)\in \mathcal D(\Omega \times (0,T)):$
\begin{equation}
\langle u,\partt\Psi\rangle + \langle u\otimes u : \nabla\Psi\rangle + \langle p,\nabla \cdot \Psi\rangle=0\label{Eulerdistribution}\,, \quad \hbox{in} \quad L^1(0,T)\,.\\\
  \end{equation}
%\end{subequations}
 %\end{document}
Assume that
\begin{equation}
u\in L^3((0,T); C^{0,\alpha}(\overline{\Omega})),
\end{equation} with $\alpha > \frac13$, then the  energy conservation holds true, that is:
\begin{equation}
\|u(.,t_2)\|_{L^2(\Omega)}= \|u(., t_1)\|_{L^2(\Omega)}\,, \quad \hbox{for every} \quad t_1, t_2 \in (0,T)\,.
\end{equation}
\end{theorem}
\noindent See also \cite{RobRod}, where the incompressible Euler equations are considered on $\T\times\R_+$ with the same impermeability condition as in the above theorem, and energy conservation is proven under an integral condition similar to that of Shvydkoy \cite{shvydkoy}.\\
We conclude this section with an Onsager-type statement for the inhomogeneous incompressible Euler system on $\T\times(0,T)$:
\begin{equation}\label{inhom}
\begin{aligned}
\partial_t(\rho u)+\diverg(\rho u\otimes u)+\nabla p&=0,\\
\partial_t\rho+\diverg(\rho u)&=0,\\
\diverg u&=0.
\end{aligned}
\end{equation}
The following theorem was proven by Feireisl et al. in \cite{FeGSGW}
\begin{theorem}\label{inhomonsager}
Let $\rho$, $u$, $p$ be a solution of~\eqref{inhom} in the sense of distributions. Assume
\begin{equation}\label{besovhypo}
u\in B_p^{\alpha,\infty}((0,T)\times\Tbb^d),\hspace{0.3cm}\rho, \rho u\in B_q^{\beta,\infty}((0,T)\times\Tbb^d),\hspace{0.3cm}p\in L^{p^*}_{loc}((0,T)\times\Tbb^d)
\end{equation}
for some $1\leq p,q\leq\infty$ and $0\leq\alpha,\beta\leq1$ such that
\begin{equation}\label{exponenthypo}
\frac{2}{p}+\frac{1}{q}=1,\hspace{0.3cm}\frac{1}{p}+\frac{1}{p^*}=1,\hspace{0.3cm}2\alpha+\beta>1.
\end{equation}
Then the energy is locally conserved, i.e.
\begin{equation}\label{localenergy}
\partial_t\left(\frac{1}{2}\rho|u|^2\right)+\diverg\left[\left(\frac{1}{2}\rho|u|^2+p\right)u\right]=0
\end{equation}
in the sense of distributions on $(0,T)\times\Tbb^d$.
\end{theorem}
Notice that, unlike for the homogeneous case, Besov regularity is assumed here also in time. This is due to the fact that the term involving a time derivative is no longer linear, hence generating a commutator estimate. This time regularity assumption can be disposed of in the vacuumless case by introducing the momentum $m=\rho u$ and formulating the equations in terms of $\rho$ and $m$, as was done in \cite{LeSh}. However then a regularity assumption is required on the pressure. A different still approach was proposed by Chen and Yu, who transfer the time regularity of the density onto the spatial regularity via the continuity equation. Thus a regularity assumption is imposed on $\nabla_x\rho$, see \cite{ChenYu} for details.\\
Further we remark that Theorem \ref{inhomonsager} implies also conservation of total energy
\[E(t) = \frac{1}{2}\int_{\T}\rho(x,t)|u(x,t)|^2\ \dx,
\]
however only in the sense of distributions. To guarantee that this quantity is the same at all times further assumptions need to be made, namely $\rho, u\in L^\infty(\T\times(0,T))$ and 
\[\sup_{t\in[0,T]}(\norm{\rho}_{B_q^{\beta,\infty}(\Tbb^d)}+\norm{\rho u}_{B_q^{\beta,\infty}(\Tbb^d)})<\infty,
\]
with $\beta >0$, cf. Corollary 3.3 in \cite{FeGSGW}.\\
Finally we state a result specific to the inhomogeneous case, which demonstrates that by strengthening regularity assumptions on the velocity it is possible to relax the assumptions on the density - and still obtain energy conservation.
\begin{corollary}\label{BV}
Let $\rho\in (BV\cap L^\infty)((0,T)\times\Tbb^d)$ and $u\in (B_3^{\alpha,\infty}\cap L^\infty)((0,T)\times\Tbb^d)$ be a solution of~\eqref{inhom}, where $\alpha>\frac{1}{3}$. Then the energy is conserved.
\end{corollary}
%%%%%%%%%%%%%%%%%%%%%%%%%
%%%%%%%%%%%%%%%%%%%%%%%%%
%%%%%%%%%%%%%%%%%%%%%%%%%
\section{Dissipative solutions}\label{secdissipative}
In the previous section we focused on the conservative direction of Onsager's conjecture. Here we will briefly discuss the possibilities of energy dissipation (or even creation) by weak solutions to the incompressible Euler equations, as well as other related systems. First however we make an observation that the case of hyperbolic systems of conservation does not pose much challenge. It is in particular well-known that there exist discontinuous shock solutions of regularity $BV\cap L^\infty$, which dissipate energy. See \cite{dafermos} for a detailed discussion. Since $BV\cap L^\infty\subset B_3^{\frac{1}{3},\infty}$, cf. \cite{FeGSGW} for proof, we can see that the main result of Section \ref{secSimple} is sharp.

The situation turned out to be much more complicated for the incompressible Euler. Here, similarly as for the simple equation ~\eqref{simpleConservlaw} the weak solutions are non-unique - as first shown by Scheffer in \cite{scheffer}, who constructed a compactly supported weak solution of class $L^2(\R^2\times\R)$. See also \cite{shnirel} for a simpler construction. Thus there is no uniqueness of weak solutions for ~\eqref{incompressEuler}, even for an identically zero initial condition. Physically, weak solutions like these constructed by Scheffer are pure nonsense. One can hardly expect water in a glass to instantaneously develop tempestuous behaviour, and then as suddenly become calm again. One might therefore hope that imposing an admissibility condition, motivated by the second law of thermodynamics, might help recover uniqueness, just as in the case of a single scalar conservation law. However this is not true either. In 2010 De Lellis and Sz\'{e}kelyhidi Jr. \cite{DLS10} have shown the following theorem regarding the system ~\eqref{incompressEuler} with initial condition $u(x,0) = u_0(x)$.
\begin{theorem}
Let $d\geq 2$. There exist bounded and compactly supported divergence-free vector fields $u_0$ for which there are
\begin{enumerate}
    \item infinitely many weak solutions satisfying both the strong and local energy inequalities.
    \item weak solutions satisfying the strong energy inequality but not the global energy equality
    \item weak solutions satisfying the weak energy inequality but not the strong energy inequality.
\end{enumerate}
\end{theorem}
By the strong and local energy inequalities it is meant that ~\eqref{localenergyintro} and ~\eqref{globalenergyintro}, respectively, are satisfied as inequalities. The weak energy inequality is defined to be
\[\int_{\R^d}\frac{1}{2}|u(x,t)|^2\ \dx = \int_{\R^d}\frac{1}{2}|u(x,0)|^2\ \dx.
\]

The construction of these solutions followed the scheme of \emph{convex integration}, which was first used in the context of fluid dynamics in \cite{DLS09}. Later dissipative weak solutions were constructed in \cite{DLS14} within the space of H\"{o}lder continuous functions, with exponent $\alpha = \frac{1}{10}$. See also \cite{DLS13}. Subsequently many authors (eg. \cite{BDLIS}, \cite{BDLS}, \cite{Ise13}) contributed efforts to increase the H\"{o}lder exponent, building upon new perturbation profiles introduced in \cite{DaneriSz}. Finally, the Onsager-critical exponent of $\frac{1}{3}$ was reached by Isett in \cite{isett}, who constructed non-conservative solutions within the class $C([0,T];C^{\frac13}(\Tbb^3))$, and Buckmaster et al. in \cite{BucDeLSzV}, where a solution of the same regularity is constructed with strictly decreasing kinetic energy. Thus Conjecture ~\ref{onsagerconjecture} is fully resolved and proven to be true.

Let us remark that the question of ill-posedness within the class of dissipative weak solutions was also asked in the context of other systems of mathematical physics. We refer the reader to \cite{ChiFeiKre}, where the equations of a compressible heat conducting gas are discussed. Further in \cite{DonFeiMar} and in \cite{CarFeiGSG} the Euler-Korteweg-Poisson system and Euler systems with non-local interactions are considered, respectively, while in \cite{FeiGSG} the Savage-Hutter model is investigated.

\section{Onsager-type statements for other systems}\label{secOnsagertype}

\subsection{Compressible Euler}
Consider now the isentropic Euler equations:
\begin{equation}\label{compressible}
\begin{aligned}
\partial_t(\rho u)+\diverg(\rho u\otimes u)+\nabla p(\rho)&=0,\\
\partial_t\rho+\diverg(\rho u)&=0.
\end{aligned}
\end{equation}
Notice that here the pressure is a given function, assumed here to depend only on the density, rather than being merely a Lagrange multiplier as in the incompressible Euler system. The local form of energy equality for this system takes the form
\begin{equation}\label{energycompressibleEuler}
\partial_t\left(\frac{1}{2}\rho|u|^2+P(\rho)\right)+\diverg\left[\left(\frac{1}{2}\rho|u|^2+p(\rho)+P(\rho)\right)u\right]=0,
\end{equation}
where $P$ is the pressure potential defined by    
\begin{equation*}
P(\rho)=\rho\int_1^\rho\frac{p(r)}{r^2}\ \mathrm{d}r.
\end{equation*}
The following theorem, proven in \cite{FeGSGW} is an analogue of Theorem \ref{inhomonsager}, requiring Besov regularity in both space and time.
\begin{theorem}\label{compressibleonsager}
Let $\rho$, $u$ be a solution of~\eqref{compressible} in the sense of distributions. Assume 
\begin{equation*}
u\in B_3^{\alpha,\infty}((0,T)\times\Tbb^d),\hspace{0.3cm}\rho, \rho u\in B_3^{\beta,\infty}((0,T)\times\Tbb^d),\hspace{0.3cm}
0 \leq \underline{\rho} \leq \rho \leq \overline{\rho} \ \mbox{a.a. in} (0,T)\times\Tbb^d,
\end{equation*}
for some constants $\underline{\rho}$, $\overline{\rho}$, and
$0\leq\alpha,\beta\leq1$ such that
\begin{equation}\label{alphabeta}
\beta > \max \left\{ 1 - 2 \alpha; \frac{1 - \alpha}{2} \right\}.
\end{equation}
Assume further that $p \in \Crm^2[\underline{\vr}, \overline{\vr}]$, and, in addition
\begin{equation}\label{pressure}
p'(0) = 0 \ \mbox{as soon as}\ \underline{\vr} = 0.
\end{equation}
Then the energy is locally conserved, i.e. ~\eqref{energycompressibleEuler} holds in the sense of distributions on $(0,T)\times\Tbb^d$.
\end{theorem}
We remark that the $\Crm^2$ assumption on the pressure is required for the machinery presented in Section \ref{secSimple} to work. This is clearly satisfied by the isentropic pressure law $p(\rho) = \kappa\rho^\gamma$ for $\gamma>1$, provided there is no vacuum (i.e. $\underline{\rho}>0$) or $\gamma>2$. Similarly as in the incompressible case, conservation of total energy can be deduced under similar additional assumptions. The following theorem in turn shows that one can dispose of the Besov regularity in time, at the expense of assuming $BV\cap\Crm$ regularity in space. 
\begin{theorem}\label{compressibleonsagerBV}
Assume that the pressure $p$ satisfies
\begin{equation} \label{pres2}
p \in C^2(0, \infty) \cap C[0, \infty), \ p(0) = 0.
\end{equation}
Let $\rho \in L^\infty((0,T) \times \T)$, $u \in L^\infty((0,T) \times \T)$ be a solution of~\eqref{compressible} in the sense of distributions. In addition, assume that
\[
u(t) \in BV \cap C(\T),\ \rho(t) \in BV \cap C(\T) \ \mbox{for a.a.}\ t \in [0,T]
\]
and
\begin{equation} \label{E1-thm}
u, \ \rho \in L^\infty(0,T; C(\T)), \ \nabla u, \ \nabla \rho \in L^\infty_{{\rm weak}-(*)} (0,T; \mathcal{M}(\T)).
\end{equation}
Then the energy is locally conserved, i.e.
\begin{equation*}
\partial_t\left(\frac{1}{2}\rho|u|^2+P(\rho)\right)+\diverg\left[\left(\frac{1}{2}\rho|u|^2+p(\rho)+P(\rho)\right)u\right]=0
\end{equation*}
in the sense of distributions on $(0,T)\times\T$.
\end{theorem}
A straightforward consequence is that the total energy
\[E(t) = \int_{\T}\frac{1}{2}\rho(x,t)|u(x,t)|^2 + P(\rho(x,t))\ \dx
\]
is constant in time (with no additional assumptions).
\subsection{Navier-Stokes}
Consider the density-dependent incompressible Navier-Stokes equations:
\begin{equation}\label{incompressNS}
\begin{aligned}
\partt (\rho u) + \diverg(\rho u\otimes u) - \mu\Delta u &= -\nabla p + \rho f,\\
\partt \rho + \diverg (\rho u) &= 0,\\
\diverg u &= 0.
\end{aligned}
\end{equation}
As before $\rho:\tplus\to\R$ and $u:\tplus\to\R^d$ are density and velocity, respectively, of a fluid and $p:\tplus\to\R$ is the pressure. Further $f:\tplus\to\R^d$ is an external force and $\mu$ is the (constant) viscosity coefficient. Similarly as for the Euler equations, if $(\rho,u)$ is a classical solution of ~\eqref{incompressNS}, then multiplying the momentum balance by $u$, one can easily see that the following energy identity is true
\begin{equation}\label{energyincomressNS}
E(t) - E(0) = - \mu \int_0^t \norm{\nabla u}_{L^2(\T)}^2\ \mathrm{d}s + \int_0^t \int_{\T} \rho u \cdot f \ \dx\mathrm{d}s,
\end{equation}
where $E(t) = \int_{\T}\frac{1}{2}\rho(x,t)|u(x,t)|^2\ \dx$.
It is well-known that, under certain assumptions, Leray-Hopf weak solutions exist globally and satisfy ~\eqref{energyincomressNS} as an inequality. However, since \emph{a priori} $u\in L^2([0,T];H^1(\T))$, dissipation of energy is possible, so that the energy equality may fail to hold. In the paper of Leslie and Shvydkoy \cite{LeSh} a sufficient condition is provided, in terms of Besov regularity, for a weak solution to satisfy ~\eqref{energyincomressNS}.
\begin{theorem}
Let $(\rho, u, p)$ be a weak solution to the density-dependent incompressible Navier-Stokes equations on $\T$, $d>1$.  Assume $(\rho, u,p)$ satisfies
\begin{align}
&u\in L^2([0,T];H^1(\T)), \ 0<\underline{\rho}\le \rho \le \overline{\rho}<\infty, \text{ and } f\in L^2(\T\times[0,T]),\\
&\rho\in L^a([0,T];B_{a}^{\frac{1}{3},\infty}),\;
u\in L^b([0,T]; B_{b}^{\frac13,c(\Nbb)}),\;
p\in L^{\frac{b}{2}}([0,T]; B^{\frac13,\infty}_{\frac{b}{2}}),\quad \frac1a + \frac3b=1,\; b \ge 3.
\end{align}
Then $(\rho, u, p)$ satisfies the energy balance relation ~\eqref{energyincomressNS} on the time interval $[0,T]$. 
\end{theorem}
Consider now the compressible Navier-Stokes equations
\begin{equation}
\begin{aligned}
\label{CNS}
\partt\rho+\diverg(\rho u)&=0\\
\partt(\rho u)+\diverg(\rho u\otimes u)+\nabla p-2\mu\Delta u-\lambda\nabla\diverg u&=0,
\end{aligned}
\end{equation}
defined on $\T\times[0,T]$, where the viscosity coefficients $\mu,\lambda$ satisfy $\mu>0$ and $2\mu+d\lambda\geq 0$.
The following theorem was proven by Yu in \cite{Yu}
\begin{theorem}
Let $(\rho, u)$ be a weak solution of \eqref{CNS} %in the sense of \cite{FNP,Lions}.
If
\begin{equation}
0\leq \rho(t,x)\leq \bar{\rho}<\infty,\;\;\text{ and }\nabla\sqrt{\rho}\in L^{\infty}(0,T;L^2(\T)),
\end{equation}
\begin{equation}
\label{condition for velocity}
u \in L^p(0,T;L^q(\T))\quad\text{ for any }  \frac{1}{p}+\frac{1}{q}\leq\frac{5}{12}, \text{ and } q\geq 6 ,
\end{equation}
and
\begin{equation}
\label{initial velocity assp}
u_0\in L^{k}(\T), \;\frac{1}{k}+\frac{1}{q}\leq \frac{1}{2},
\end{equation}
 then such a weak solution
 $(\rho, u)$  satisfies the following energy identity
\begin{equation*}
\begin{split}
&\int_{\T}\left(\frac{1}{2}\rho|u|^2+\frac{\rho^{\gamma}}{\gamma-1}\right)\ \dx+2\mu\int_0^T\int_{\T}|\nabla u|^2\ \dxdt
\\&\quad\quad\quad\quad\quad\quad+\lambda\int_0^T\int_{\T}|\diverg u|^2\ \dxdt=
\int_{\T}\left(\frac{1}{2}\rho_0|u_0|^2+\frac{\rho_0^{\gamma}}{\gamma-1}\right)\ \dx
\end{split}
\end{equation*}
 for any $t\in [0,T].$
\end{theorem}

\subsection{Magnetohydrodynamics}
We consider now the equations of ideal magnetohydrodynamics (MHD for short) in three spatial dimensions.
\begin{equation}\label{ideal MHD}
\begin{aligned}
\partial_t u + \diverg(u\otimes u)  &= -\nabla p -\frac12\nabla|b|^2 + b\cdot\nabla b,\\
\partial_t b + \diverg(b\otimes b) &= b \cdot \nabla u,\\
\diverg u & = \diverg b = 0.
\end{aligned}
\end{equation}
where $b$ is the magnetic field acting on the fluid.
A straightforward adaptation of the method used in \cite{ConstETiti} was emplyed by Caflisch et al. \cite{Cafetal} to prove the following theorem on global energy conservation for this system defined on the torus.
\begin{theorem}\label{caflisch}
Let $d = 2,3$ and let $(u,b)$ be a weak solution of ~\eqref{ideal MHD}. Suppose that
\[u\in C([0,T],B_3^{\alpha,\infty}(\T)),\;\; b\in C([0,T],B_3^{\beta,\infty}(\T))
\]
with
\[\alpha > \frac{1}{3},\;\;\alpha+2\beta > 1.
\]
Then the following energy identity holds for any $t\in[0,T]$
\begin{equation}\label{energyMHD}
\int_{\T}|u(x,t)|^2 + |b(x,t)|^2\ \dx = \int_{\T}|u(x,0)|^2 + |b(x,0)|^2\ \dx.
\end{equation}
\end{theorem}
This result was later extended by Kang and Lee \cite{KangLee}, who used an approach similar to that of Cheskidov et al.
\begin{theorem}
let $(u,b)$ be a weak solution of ~\eqref{ideal MHD}. Suppose that
\[u\in L^3([0,T],B_3^{\alpha,c(\Nbb)}(\R^3)),\;\; b\in L^3([0,T],B_3^{\beta,c(\Nbb)}(\R^3))
\]
with
\[\alpha \geq \frac{1}{3},\;\;\alpha+2\beta \geq 1.
\]
Then ~\eqref{energyMHD} holds.
\end{theorem}
Similar results can be proven regarding conservation of magnetic helicity and cross-helicity, see \cite{KangLee} for details.

\section{General conservation laws}\label{secgeneral}
The reader will have noticed similarities in the statements regarding sufficient regularity conditions guaranteeing energy/entropy conservation for the aforementioned systems of equations of fluid dynamics. Especially the differentiability exponent of $\frac{1}{3}$ is a recurring condition. One might therefore anticipate that a general statement could be made, which would cover all the above examples and more. Indeed it is possible, as shown recently by Gwiazda et al. in \cite{GMSG}. Using nomenclature and notation as in \cite{dafermos}, they consider a general conservation law of the form
\begin{equation}\label{generalconservlaw}
\diverg_X(G(U(X))) = 0.
\end{equation}
Here $U:\Xcal\to\Ocal$ is an unknown vector-valued function and $G:\Ocal\to\Mbb^{n\times(d+1)}$ is a given matrix field, where $\Xcal$ is an open subset of $\R^{d+1}$ or $\T\times\R$ and the set $\Ocal$ is open in $\R^n$. It is easy to see that any classical solution to ~\eqref{generalconservlaw} satisfies also
\begin{equation}\label{companionlaw}
\diverg_X(Q(U(X))) = 0,
\end{equation}
where $Q:\Ocal\to\R^{s\times(d+1)}$ is a smooth function such that
\begin{equation}\label{companionrelation}
D_UQ_j(U) = \Bfrak(U)D_UG_j(U),\;\;\text{for all}\;U\in\Ocal,\;j\in{0,\cdots,k},
\end{equation}
for some smooth function $\Bfrak:\Ocal\to\Mbb^{s\times n}$.
The function $Q$ is called a \emph{companion} of $G$ and equation ~\eqref{companionlaw} is called a \emph{companion law} of the conservation law ~\eqref{generalconservlaw}. In applications mentioned in previous sections the companion law would be the energy equality. 
The following theorem, proved in \cite{GMSG}, answers the question of how much regularity of a weak solution to ~\eqref{generalconservlaw} is required so that it also satisfies the companion law ~\eqref{companionlaw}.
\begin{theorem}
Let $U\in B_3^{\alpha,\infty}(\Xcal;\Ocal)$ be a weak solution of ~\eqref{generalconservlaw} with $\alpha>\frac13$. Assume that $G\in\Crm^2(\Ocal;\Mbb^{n\times(d+1)})$ is endowed with a companion law with flux $Q\in \Crm(\Ocal;\Mbb^{s\times (d+1)})$ for which there exists  $\Bfrak\in\Crm^1(\Ocal;\Mbb^{s\times n})$ related through identity \eqref{companionrelation} and the essential image of $U$ is compact in ~$\Ocal$.\\
Then $U$ is a weak solution of the companion law \eqref{companionlaw}  with the flux $Q$.
\end{theorem}
Notice that, perhaps not surprisingly, the generality of the above theorem is achieved at the expense of optimality of the assumptions. Given additional information on the structure of the problem at hand one might be able to relax some of these assumptions, as discussed in the previous sections. Let us mention that the theorem provides for instance a conservation of energy result for the system of polyconvex elastodynamics.

\section{Acknowledgements}
This work was partially supported by the Simons -- Foundation grant 346300 and the Polish Government MNiSW 2015-2019 matching fund.
T.D acknowledges the support of the National Science Centre, DEC-2012/05/E/ST1/02218. The research was partially supported by the Warsaw Center of Mathematics and Computer Science. P.G and A.\'S-G received support from the National Science Centre (Poland), 2015/18/M/ST1/00075.

\end{document}